\documentclass[12pt,reqno]{article}
\usepackage{color, amsmath, amsfonts, amstext, amsthm, latexsym, mathenv, hyperref}
\allowdisplaybreaks
\usepackage{amssymb}%
\usepackage{etoolbox}%
\newcommand{\bb}[1]{%
  \renewcommand*{\do}[1]{%
    \expandafter\newcommand\csname ##1\endcsname{\ensuremath{\mathbb{##1}}}%
  }%
  \docsvlist{#1}%
}

\newcommand{\veps}{\varepsilon}

\newcommand{\F}{{\mathbb F}}

\newcommand{\N}{{\mathbb N}}
\newcommand{\Q}{{\mathbb Q}}
\newcommand{\R}{{\mathbb R}}
\newcommand{\E}{{\mathbb E}}
\newcommand{\C}{{\mathbb C}}
\newcommand{\PP}{{\mathbb P}}

\newcommand{\Ac}{{\mathcal  A}}

\newcommand{\Fc}{{\mathcal  F}}
\newcommand{\Oc}{{\mathcal  O}}
\newcommand{\Uc}{{\mathcal  U}}

\newcommand{\wlth}{\mathfrak{w}}
\newcommand{\Wlth}{\mathfrak{W}}

\newcommand{\trace}{\mathrm{Tr}}

\numberwithin{equation}{section}
\newtheorem{theorem}{Theorem}[section]

\newtheorem{lemma}[theorem]{Lemma}
\newtheorem{remark}[theorem]{Remark}
\newtheorem{prop}[theorem]{Proposition}
\newtheorem{coro}[theorem]{Corollary}

\newtheorem{condition}{Condition}
 \def\vs#1{\vspace{#1ex}}

\newcommand{\comment}[1]{}

\newcommand{\datum}{Version 2014-05-14} \date{August 2013, \datum \\  \vs2 Dedicated to the 70th birthday of Ivar Ekeland}

\author{Yal\c{c}in Aktar\thanks{\small Chair in Mathematical Finance EISTI and CMAP, Ecole Polytechnique Paris,  \texttt{yar@eisti.eu}} \and Erik Taflin\thanks{\small Chair in Mathematical Finance EISTI and AGM,  Universit\'e de Cergy,  \texttt{taflin@eisti.fr}}}

\title{A remark on smooth solutions to a stochastic control problem  with a power terminal cost function and stochastic volatilities\thanks{The authors thank Nizar Touzi for having drawn their attention to the topic of this article and many constructive discussions.}}

\begin{document}

\maketitle

\begin{abstract} 
Incomplete financial markets are considered,  defined by a multi-dimensional non-homogeneous diffusion process, being the direct sum of an It\^{o} process (the price process), and another  non-homogeneous diffusion process (the exogenous process, representing exogenous stochastic sources). The drift and the diffusion matrix of the price process are functions of the time, the price process itself and the exogenous process.

In the context of such markets and for power utility functions, it is proved that the stochastic control problem consisting of optimizing the expected utility of the terminal wealth, has a classical solution (i.e. $C^{1,2}$).

This result paves the way to a study of the optimal portfolio problem in incomplete forward variance stochastic volatility models, along the lines of Ekeland et al. \cite{EkTa2005}.
\end{abstract}


\vs3


\section{Introduction} \label{sec: intro}
The seminal papers \cite{Mert69} and \cite{Mert71}, of Merton, on portfolio optimization in continuous time, are formulated for a financial market where the multidimensional (spot) price is a Markov process. In particular, when the  price is a non-homogeneous diffusion, the Hamilton-Jacobi-Bellman equation was derived, and it  was solved  explicitly in the special case of a log-normal price process and a power utility function. Important generalizations of Merton's work have been accomplished by the study of the regularity of viscosity solutions to the HJB equation and by the use of duality methods.

The purpose of this article is to solve the optimal portfolio problem,  in the context of the incomplete markets models \textbf{(FM)} defined below, with ``stochastic volatility'', by first establishing  that the HJB equation has a classical solution and then applying a verification theorem:  
\begin{itemize}
\item \textbf{Financial Market (FM)} \\  There is one tradeable risk-free asset with vanishing interest rate. There are exactly $n$ tradeable basic risky assets in the market and $X$ is a  $n$-dim. It\^{o} process, whose coordinates are (a simple function of) the spot prices of the tradeable basic risky assets. Exogenous stochastic sources are represented by a $d$-dim. non-homogeneous diffusion process $Y$ with the diffusion matrix being invertible. The process $(X,Y)$ is a non-homogeneous diffusion process,  so the drift and the diffusion matrix of $X$ are functions of the time $t$, $X_t$ (the price process at $t$) and $Y_t$ (the exogenous process at $t$)
\end{itemize}
The above  optimal portfolio problem will only be considered for power utility functions, defining the bequest, without consumption and in its unconstrained version, in the sens that the portfolio is allowed to take any value in $\R^{1+n}$.  It will be referenced by \textbf{(OPP)}. The comments below concerning related papers, only refer to this unconstrained case.

In the case of power utility functions, generalizations of Merton's framework to incomplete markets with stochastic volatility, not necessarily of the above type \textbf{(FM)},  by proving that the HJB equation has a classical solution, have been studied by many authors under various hypotheses, cf. \cite{Zarip 2001}, \cite{Pham 2002}, \cite{Lindberg Stoch-Vol 2006}, \cite{Delong-Kluppelberg 2008}, \cite{Berdj-Perg 2013} and references therein. In these works, the coefficients of the SDEs defining the model are independent of price $X_t$.

The markets models in \cite{Zarip 2001}, \cite{Pham 2002}, \cite{Berdj-Perg 2013} are  of the type \textbf{(FM)}. In \cite{Zarip 2001} the pure investment problem (i.e. only a power bequest function and  no consumption) is considered,  $n=d=1$, which permits to transform the HJB equation into a linear PDE, by a fractional substitution and obtain explicit solutions. In \cite{Pham 2002} the pure investment problem is considered, $n$ and $d$ are arbitrary, the coefficients of the model are independent of time and by an exponential substitution the HJB equation is transformed  into a semi-linear PDE,  which is proved to have a classical solution. A necessary and sufficient condition is given by \cite[Remark 3.1, Eq. (3.7)]{Pham 2002} for the existence of fractional substitution transforming the HJB equation into a linear equation, like in \cite{Zarip 2001}. Reference  \cite{Berdj-Perg 2013} considers the problem with consumption and bequest function for general $n \geq d$ and the coefficients of the model are restricted to satisfy (after a simple transformation) the condition \cite[Remark 3.1, Eq. (3.7)]{Pham 2002}, here with time dependent coefficients. A fractional substitution then transforms the HJB equation into a semi-linear equation, which is proved to have a classical solution.

In the works \cite{Lindberg Stoch-Vol 2006}, with general $n$,   and \cite{Delong-Kluppelberg 2008}, with $n=1$, the exogenous process is of Ornstein-Uhlenbeck type, driven by a subordinator with c\`{a}dl\`{a}g sample paths, i.e. a L\'{e}vy process with a.s. non-decreasing sample paths. The drifts and volatilities are time-independent. Reference \cite{Castaneda-Hernandez 2005} is based on duality methods.

Our main motivation, for this work, is to find a solution of the portfolio problem \textbf{(OPP)},  useful for solving corresponding optimal portfolio problems in the framework of incomplete forward variance stochastic volatility models. Such models have a natural formulation in an infinite dimensional setting close to what is used for Zero-Coupon Bond markets and the dynamics of forward rate curves, cf. for theoretical developments  \cite{Bj-Sv01}, \cite{Buehler 2006},    \cite{EkTa2005},  \cite{Filipovic Teichmann 2004} and for applications  \cite{Bergomi-Guyon 2012} and references therein. This requires an extension of earlier works on  the portfolio problem \textbf{(OPP)} to markets of the  type \textbf{(FM)}, where the coefficients of the models are allowed to be functions also of time and asset prices.

We have accomplished this,  first by extending the framework of \cite{Pham 2002} to cover the case of market models of the type \textbf{(FM)}, %
and then by solving the portfolio problem \textbf{(OPP)} in this context.  Although our problem is more complex, the main ideas of \cite{Pham 2002} can be adapted to our proofs. Also in our case a standard substitution (see (\ref{candsol})) transforms the HJB equation into a semilinear second order PDE (\ref{vHJB2}), quadratic in the first derivatives. This PDE is regularized by, decreasing the quadratic growth to a linear growth of the Hamiltonian in its first order derivative variable. This corresponds to a multiplication of the  Legendre-Fenchel transformed Hamiltonian by  a cut-off function. The existence and uniqueness (see Lemma \ref{lemma:existence u^R}) of a solution to the regularized HJB equation (\ref{HJB:on:u_k}) follows from a standard result \cite[Theorem 6.2 Chap VI]{Flem Rish 75}. After a reformulation in terms of a stochastic control problem, a crucial estimate, uniform in the cut-off, of the derivative of regularized solution is obtained, Lemma \ref{lemma:linear:growth:u^k}. The proof of this lemma is based on Appendix \ref{app: A}, which generalizes  \cite[Lemma 11.4 ]{Flem Soner 1993} to our case. The uniform estimate of the derivative permits to prove the convergence of the regularized solution to a solution of the semilinear PDE, when the cut-off ``disappears'', Theorem \ref{thm:global}. A verification result is then proved using elementary properties of the Girsanov transformation, which gives the main result Theorem \ref{thm: main}.

Certain differences in the hypothesis of this work and \cite{Pham 2002} are due to that some growth conditions of model coefficients, announced as linear growth in \cite{Pham 2002} should be replaced by square rote growth (see point 4. of Remark \ref{rmk: exists S Y 1}).

\vs2

We finnish this long introduction by setting some of the notations to be used. \\
\textbf{Notations:} \\
For linear spaces $F$ and $G$,  $L(F,G)$ is the linear space of linear continuous operators of $F$ into $G$.  $L(F,G)$ is endowed with the operator norm.

A linear operator and its matrix representation  (w.r.t. a given orthonormal basis) will not be distinguished.  $A'$ denotes the adjoint operator of a linear operator $A$ and $|A|$ the operator norm of $A$.

Let $n\geq 1$ be an integer and  for $i \in \{0, \ldots , n\}$ let $\Oc_i$ be open subsets of finite dimensional vector spaces. For a function $f: \Oc_1 \times \cdots \times \Oc_n \rightarrow \Oc_0$, when well-defined, the partial derivative of order $m$ w.r.t. variables $x_1, \ldots, x_m$, where for every $i$,  $x_i \in \Oc_j$ for some $j$, is denoted $f_{x_1\ldots x_m}$. To avoid confusions, occasionally we  write $\nabla_x f$ instead of $f_x$ etc.

If not stated otherwise:  $z=(x,y) , r=(p,q)\in E=\R^{n} \times \R^{d}$.

The identity function is denoted $\mathrm I$, possibly with an index indicating in which set.

Scalar product of $a, b \in F$, depending on the context, $a \cdot b$,  $(a,b)_F$ and $(a,b)$ are used.

The open ball of radius $R$ centered at $0$ in a normed space $E$ is denoted  $B_E(0,R)$ and $\bar{B}_E(0,R)$ is its closure (or just $B(0,R)$ and $\bar{B}(0,R)$).

\section{The mathematical model and main result} \label{sec: model}

Let $B$ and $W$ be two independent standard Brownian motions, of dimension $m$ and $d$ respectively, restricted to a time interval $[0,T]$ (with $T>0$), on a complete probability space $(\Omega,\Fc,\PP)$ endowed with the complete filtration $\F=(\Fc_t)_{0\leq t\leq T}$ generated by $\tilde W=  (B,W)$.

We consider, on the time interval $[0,T],$ a financial market with $n$ risky assets, whose price processes  $S^i$, $1\leq i \leq n$ are strictly positive, and one risk-free asset whose interest rate is $0$. Since  $S^i$ is  strictly positive, we can express the dynamics in terms of the $\R^n$-valued process $X$, whose $i$:th coordinate is $X^i=\ln(S^i)$. The process $X$ is supposed to satisfy the SDE
\begin{equation} \label{dXt}
dX_t=\tilde{\mu}_1(t,Z_t)\,dt+\sigma_1(t,Z_t)\,dB_t + \sigma_2(t,Z_t)\,dW_t, \;\; X_0 \in \R^n,
\end{equation}
where $Z_t=(X_t,Y_t)$ is an $E=\R^{n} \times \R^{d}$ valued process and $Y$ is a  $\R^d$-valued process representing the ``exogenous stochastic factors'' of the model and is supposed to  satisfy the SDE
\begin{equation}\label{dYt}
dY_t=\mu_2(t,Y_t)\,dt+dW_t, \;\; Y_0 \in \R^d.
\end{equation}
In equation (\ref{dXt}), $\tilde{\mu}_1$, $\sigma_1$ and $\sigma_2$ are continuous functions of $[0,T] \times E$ into $\R^n$, $L(\R^m,\R^n)$ and  $L(\R^d,\R^n)$ respectively.  In equation (\ref{dYt}), $\mu_2$ is a  continuous function of $[0,T] \times \R^d$ into  $\R^d$.

We introduce the  volatility functions  $\sigma :[0,T] \times E \rightarrow  L(\R^{m+d},\R^n)$  of  the   SDE (\ref{dXt}) and $\Sigma :[0,T] \times E \rightarrow L(\R^{m+d},E)$ of  the system of SDEs (\ref{dXt}) and (\ref{dYt}) by, for all $t \in [0,T]$, $z \in E$,  $a=(b,c)$, $b \in \R^{m}$ and $c \in \R^{d}$
\begin{equation}\label{def: vol fnct}
\sigma(t,z)a= \sigma_1(t,z)b + \sigma_2(t,z)c \text{ and }
  \Sigma(t,z)a= (\sigma(t,z) a,  c). %
\end{equation}
We also introduce the functions $\sigma^i :[0,T] \times E \rightarrow \R^{m+d}$, $1 \leq i \leq n$, by
$$\sigma^i(t,z)= (\sigma^{i1}(t,z), \ldots, \sigma^{i \, m+d}(t,z))$$
and the drift function $\mu_1:  [0,T] \times E \rightarrow \R^n$ by
$$\mu_1= \tilde{\mu}_1 + \beta, \text{ where } \beta=(\beta^1, \ldots, \beta^n), \; \beta^i(t,z)=\frac{1}{2} | \sigma^i(t,z) |^2,$$
and the norm is the Euclidean norm in $E$.

The coefficients in (\ref{dXt}) and (\ref{dYt}) are supposed to satisfy various conditions for different purposes.

Conditions related to existence of $Z$:

\begin{condition} \ \label{cond: A}
\begin{itemize}
\item[\ref{cond: A}$_1)$]
If $f$ is the function  $\tilde{\mu}_1: [0,T] \times E \rightarrow  \R^n$ or $\sigma :[0,T] \times E \rightarrow  L(\R^{m+d},\R^n)$ then
\begin{equation} \label{eq: 2 lip1}
f  \in C^{1}([0,T] \times E) \text{ and } \exists \;  C \text{ s.t. }  \forall (t,z) \in [0,T] \times E \;\; |f_z(t,z)| \leq C .
\end{equation}
\item[\ref{cond: A}$_2)$]
The function  $\mu_2$ satisfies:
\begin{equation} \label{eq: 3 lip1}
 \mu_2  \in C^{1}([0,T] \times \R^d; \R^d) \text{ and } \exists \;  C \text{ s.t. }  \forall (t,y) \in [0,T] \times \R^d \;\; |\nabla_y \mu_2(t,y)| \leq C .
\end{equation}
\end{itemize}
\end{condition}

Conditions related to existence of a $C^{1,2}$ solution of the HJB equation:
\begin{condition} \ \label{cond: B}
\begin{itemize}
\item[\ref{cond: B}$_1)$] $C^{1,2}$ regularity of $\sigma$,
\begin{equation} \label{cond: sigma C2}
\sigma \in C^{1,2}([0,T] \times E, L(\R^{m+d},\R^n)),
\end{equation}

\item[\ref{cond: B}$_2)$] For  $(t,z) \in [0,T] \times E $ and for $i=1,2$, let  $M_i(t,z)=\sigma_i(t,z)\sigma_i(t,z)'$, $M(t,z)=M_1(t,z)+M_2(t,z)$. We shall consider the following conditions:
\begin{equation} \label{cond: sigma bound}
 \exists \; C \in \R \;\; : \;\; \forall \;   (t,z) \in [0,T] \times E ,  \;\; |\sigma(t,z)|   \leq C,
\end{equation}
\begin{equation} \label{cond: M1(z) inv}
\forall \;     (t,z) \in [0,T] \times E ,  \;\; M_1(t,z)^{-1}   \text{ exists and }\exists \; C \in \R \;\; : \;\; \forall \;   (t,z),  \;\; |M_1(t,z)^{-1}|   \leq C,%
\end{equation}
\item[\ref{cond: B}$_3)$]
The functions $\tilde \mu_1$ and $\mu_2$ %
satisfy the following square-root  growth condition:  There exists  $C \in  \R$ such that for all $(t,z) \in [0,T] \times E$ 
\begin{equation} \label{eq: growth}
  |\tilde \mu_1(t,z)|  \leq C \sqrt{1+|y|} \text{  and }   |\mu_2(t,y)| \leq C \sqrt{1+|y|} .  %
\end{equation}

\item[\ref{cond: B}$_4)$]  For all $(t,z)  \in  [0,T] \times E$,  $M(t,z)$ is invertible and if $f(t,z)=|M(t,z)^{-1/2}\mu_1(t,z)|^2$ then $f$ satisfies
\begin{equation} \label{cond: bound mu1 2}
f  \in C^{0,1}([0,T] \times E) \text{ and } \exists \;  C \text{ s.t. }  \forall (t,z) \in [0,T] \times E \;\; |f(t,z)|+|f_z(t,z)| \leq C .
\end{equation}
\end{itemize}
\end{condition}

Condition related to the application of a verification theorem:
\begin{condition} \ \label{cond: C}
\begin{itemize}
\item[\ref{cond: C}$_1)$]  For all $(t,z)  \in  [0,T] \times E$,  $M(t,z)$ is invertible and if $f(t,z)=\sigma_2(t,z)'M(t,z)^{-1}\mu_1(t,z)$, then $f$ satisfies
\begin{equation} \label{cond: bound mu1 3}
f  \in C^{0,1}([0,T] \times E) \text{ and } \exists \;  C \text{ s.t. }  \forall (t,z) \in [0,T] \times E \;\; |f_z(t,z)| \leq C \sqrt{1+|y|}.
\end{equation}
\end{itemize}
\end{condition}

It is standard that the system of SDE (\ref{dXt}) and (\ref{dYt}) has a unique strong solution,  when \ref{cond: A}$_1)$ and   \ref{cond: A}$_2)$ of Condition \ref{cond: A} are satisfied, cf. \cite{Touzi OSC 2013}.

When $M(t,z)$ is invertible, we shall use the notation
\begin{equation} \label{def: N}
  N(t,z)=\sigma_2'(t,z)M(t,z)^{-1}\sigma_2(t,z).
\end{equation}
One immediately obtains the following result, where point $1$ permits to simplify equation  (\ref{dXt}) (see Remark \ref{rmk: vol sqr M}).
\begin{lemma} \ \label{lm: derived conditions}
\begin{enumerate}
\item  Suppose that  $\sigma$ satisfies the conditions (\ref{eq: 2 lip1}),  (\ref{cond: sigma C2}), (\ref{cond: sigma bound}) and (\ref{cond: M1(z) inv}).  Then these conditions  are also satisfied with $\sigma$ replaced by $\tilde \sigma = (M_1^{1/2}, \sigma_2)$
 \item If $\sigma$ satisfies conditions (\ref{cond: sigma bound}) and (\ref{cond: M1(z) inv}), then
\begin{equation} \label{eq: cond A invert}
\exists \; C \in (0,1)  \;\; : \;\; \forall \;   (t,z) \in [0,T] \times  E,  \;\; 1- |N(t,z)| \geq C.
\end{equation}
\end{enumerate}
\end{lemma}
\noindent \textbf{Proof:}

\noindent $1$. According to  (\ref{cond: sigma bound}) and (\ref{cond: M1(z) inv}) there exist $0<c<C$ such that the spectrum of $M_1(t,z)$ is a subset of $(c,C)$ for all $(t,z)$. Denote by  $\C_+$  the set of complex numbers with  real part $\geq 0$. Let $\gamma$ be a simple positively oriented continuous closed curve in $\mathring{\C}_+$,  the interior of $\C_+$,  enclosing  $[c,C]$. The square root function is holomorphic in  $\mathring{\C}_+$, so by the Dunford-Taylor integral, $\forall$ $(t,z) \in [0,T] \times E$,
\begin{equation} \label{eq: Dunford-Taylor}
 M_1(t,z)^{1/2}=-\frac{1}{2\pi i} \int_{\gamma} \zeta^{1/2} R(M_1(t,z),\zeta) \, d\zeta,
\end{equation}
where $R(M_1(t,z),\zeta)=(M_1(t,z)-\zeta I)^{-1}$ is the the resolvent of $M_1(t,z)$ at $\zeta$. Using that the derivative w.r.t. $l$ being one of the variables $t$ and $z_i$ is given by
\[
 \frac{\partial}{\partial l}  R(M_1(t,z),\zeta) =- R(M_1(t,z),\zeta)  (\frac{\partial}{\partial l} M_1(t,z)) R(M_1(t,z),\zeta),
\]
one verifies that the conditions (\ref{eq: 2 lip1}),  (\ref{cond: sigma C2}), (\ref{cond: sigma bound}) and (\ref{cond: M1(z) inv}) are satisfied by $\tilde \sigma = (M_1^{1/2}, \sigma_2)$.

\noindent $2$. According to (\ref{cond: M1(z) inv}) there exists $c >0$ such that for all $(t,z) \in [0,T] \times E$ and $a \in \R^n$ we have  $(a,M(t,z)a) \geq c|a|^2 +|\sigma_2(t,z)'a|^2$ or equivalently  $|a|^2 \geq c|M(t,z)^{-1/2}a|^2 +|\sigma_2(t,z)'M(t,z)^{-1/2}a|^2$. By (\ref{cond: sigma C2}) it then exists $C > 0$ such that $|\sigma_2(t,z)'M(t,z)^{-1/2}|^2 \leq 1-c/|M(t,z)|\leq 1-c/C$ for all such $(t,z)$. Then by the definition of $N$, $|N(t,z)|=|M_1(t,z)^{-1/2}\sigma_2(t,z)|^2=|\sigma_2(t,z)'M(t,z)^{-1/2}|^2 \leq 1-c/C$.

\qed

\begin{remark}\label{rmk: vol sqr M} \textnormal{
We will use that, according to 1. of Lemma \ref{lm: derived conditions}, the volatility function $\sigma_1$ in (\ref{dXt}) can be replaced by $M_1^{1/2}$, when the function $M_1^{-1}$ exists, i.e.
 \begin{equation*}  \tag*{(\ref{dXt}')} %
dX_t=\tilde{\mu}_1(t,Z_t)\,dt+M_1(t,Z_t)^{1/2}\,d\tilde B_t + \sigma_2(t,Z_t)\,dW_t, \;\; X_0 \in \R^n,
\end{equation*}
where
\[
 \tilde B_t=\int_0^t M_1(s,Z_s)^{-1/2}\sigma_1(t,Z_s)\,dB_s.
\]
In fact, by L\'{e}vy's characterization theorem  $(\tilde B,W)$ is a standard $(\Omega,\Fc,\F,\PP)$ B.m. and one establish that: \\
\indent If conditions \ref{cond: A}$_1)$ and   \ref{cond: A}$_2)$ are satisfied and $M^{-1}(t,z)$ exists for all  $(t,z) \in [0,T] \times  E$, then $Z$ is a strong solution  to the system (\ref{dXt}) and (\ref{dYt}) iff $Z$ is a strong solution  to the system (\ref{dXt}') and  (\ref{dYt}).
}
\end{remark}

\begin{remark} \label{rmk: exists S Y 1} \textnormal{
\begin{enumerate}
\item The the volatility function $\Sigma$ in (\ref{def: vol fnct}) has for all $(t,z) \in [0,T] \times E$ the property: $\Sigma(t,z)$ is on-to $E$ iff $\sigma_1(t,z)$ is on-to, or equivalently
$$
\Sigma(t,z) \text{ is onto } E \text{ iff }   M_1(t,z)=\sigma_1(t,z)\sigma_1(t,z)' \text{ is invertible.}
$$
When $M_1(t,z)$ is invertible, one has
\[
 |(\Sigma(t,z)\Sigma(t,z)')^{-1/2}| \leq 1+|M_1(t,z)^{-1/2}|+|\sigma_2(t,z)' M_1(t,z)^{-1/2}|.
\]
 In particular $\Sigma$ satisfies the uniform parabolicity condition iff the right hand side of this inequality is uniformly bounded on $[0,T] \times E$ .
\item At several occasions, \cite[Theorem 6.2, p169]{Flem Rish 75} will be used to prove important intermediary  results of this paper and in particular for the existence of solutions of eq. (\ref{HJB:on:u_k}). A literal  application requires $m=n$.
However, if $\Sigma$ is supposed to satisfy the uniform parabolicity condition,   the case $m>n$ can be reduced to the case  $m=n$ by a redefinition of the B.m., as in Remark \ref{rmk: vol sqr M}.
 \item Let  $M(t,z)$ be invertible and let the linear operator $A(t,z)$ be as in (\ref{def: A(z)}).  One easily establish that if one of the linear operators   $M_1(t,z)$,  $\text{I} -N(t,z)$ or  $A(t,z)$ has an inverse then all three are invertible.
 \item The optimal portfolio problem for a power utility function, in case of coefficients independent of time and  price in the system of SDE (\ref{dXt}) and (\ref{dYt}), was treated in \cite{Pham 2002} (in particular see \cite[(\textbf{H3a}), p.66]{Pham 2002}). For this case there is an important difference between our hypotheses and those of \cite{Pham 2002}.  In fact, (\ref{eq: growth}) of  \ref{cond: B}$_2)$ imposes  a square-root  growth condition on $\mu_2$, but the corresponding assumption  (\textbf{H3a}i) of \cite{Pham 2002} only imposes a linear  growth condition on $\mu_2$. However, the proof of \cite[Lemma 4.1]{Pham 2002} is wrong under the linear  growth condition (\textbf{H3a}i) and it turns out that our square-root  growth condition is exactly what is needed to make it correct. 
\end{enumerate}
}
\end{remark}

It was noted in \cite[Remark 2.2]{Pham 2002} that, when $\mu_2$ satisfies the Lipschitz condition  (\ref{eq: 3 lip1}) of Condition \ref{cond: A}, then $|Y|$ satisfies the following  linear growth condition,   $\exists C \in \R$ such that $\forall t \in [0,T]$
\begin{equation}
\label{maj|Y_t|}
|Y_t|\leq C\left(1+\int_0^t|W_u|\,du+|W_t|\right),
\end{equation}
and consequently that there exists $\veps>0$ such that
\begin{equation}
\E\left[\sup_{0\leq t\leq T}e^{\veps|Y_t|^2}\right]<+\infty. \label{supEYt}
\end{equation}

Consider now an agent with a power utility function $U$ and with a self-financing investment policy $\pi$ in the above financial market,  whose purpose is to  optimize the expected utility $\E[U(\Wlth_T)]$ of the final wealth $\Wlth_T$. More precisely, suppose that the liquidation value $\Wlth_t >0$, at time $t \in [0,T]$, satisfies
\[
 d\Wlth_t=\Wlth_t\sum_{1\leq i \leq n}\pi_t^i \frac{1}{S^i_t} dS^i_t,
\]
where   $\pi=(\pi_t^1,\ldots,\pi_t^n)$ is an adapted $\R^n$-valued process. Consequently,    $\pi_t^i$ is the fraction of $\Wlth_t$ held in asset nr. $i$, $1 \leq i \leq n$.  The set $\Ac$ of admissible controls is here  (following \cite[formula (2.5)]{Pham 2002})   the set of all   $\F$ progressively measurable $\R^n$-valued processes $\pi$ satisfying
\begin{equation}
\label{intcon1}
\exists \, \veps>0 \text{ such that }  \sup_{0\leq t \leq T} \E\left[e^{\veps|\sigma(Z_t)'\pi_t|^2}\right]<\infty.
\end{equation}
For a given power $a\in(-\infty,1) \setminus\{0\}$, the utility function $U$ is defined by
\begin{equation}
U(\wlth)=\frac{\wlth^a}a,~~ \wlth > 0.
\end{equation}
The gain function $J$ is defined on $ [0,T]\times (0,\infty) \times  E \times \Ac$ by
\begin{equation}\label{eq: gain fnct}
 J(t,\wlth,z,\pi)=\E[U(\Wlth_T)|\Wlth_t=\wlth,Z_t=z],
\end{equation}
where the process $Z=(X,Y)$ satisfies the SDEs (\ref{dXt}) and (\ref{dYt}) and the process $\Wlth$ satisfies the controlled SDE
\begin{equation} \label{dWltht}
d\Wlth_s=\Wlth_s(\pi_s'\mu_1(s,Z_s)\,ds+\pi_s'\sigma(s,Z_s)\,d\tilde{W}_s), \;\; \Wlth_t=\wlth, \;\; t\leq s \leq T.
\end{equation}
The agent's optimization problem is then formulated as the following stochastic optimal control problem,  with value function $v$: $\forall \; (t,\wlth,z)\in [0,T]\times (0,\infty) \times  E$
\begin{equation}
\label{defvalfunc}
v(t,\wlth,z)=\sup_{\pi\in\Ac} J(t,\wlth,z,\pi).
\end{equation}

Here,  as in the seminal work of Merton, if a solution $v$ of equation (\ref{defvalfunc}) exists, then  the (positive) homogeneity of the utility function $U$ and the linearity of equations (\ref{dWltht}) in the variable $\Wlth$ gives directly  that the function  $\wlth \mapsto v(t,\wlth,z)$ must be homogeneous of degree $a$, i.e.
\begin{equation}
 \label{eq: homogeneity a}
\forall \; \wlth >0, \;\;  v(t,\wlth,z)=\wlth^a v(t,1,z).
\end{equation}

We note (cf. \cite[Remark 1.2]{Pham 2002}) that if $f: [0,T]\times E \rightarrow \R^n$ is Borel measurable and if there is $C \in \R$,  such that for all $t \in [0,T]$ and $z=(x,y) \in E$ one has $|\sigma(t,z)'f(t,z)| \leq C (1+|y|)$, then the control $\pi\in \Ac$, when $\pi_t=f(t,Z_t)$. This is obtained by  (\ref{supEYt}).

Finally we state hare  the main result of this paper, which is a corollary of Proposition \ref{prop:verification} and  to be proved in the following sections:
\begin{theorem}\label{thm: main} If Condition \ref{cond: A}, Condition \ref{cond: B} and Condition \ref{cond: C} are satisfied, then
\begin{enumerate}
 \item the value function $v$, defined by formula (\ref{defvalfunc}), satisfies (\ref{eq: homogeneity a}) and $v(\cdot,1,\cdot) \in C^{1,2}([0,T] \times  E)$,
\item there is a unique optimal control $\hat{\pi} \in \Ac$.
\end{enumerate}
\end{theorem}

\section{The semi-linear HJB Equation} \label{sec: semilinear HJB}

The  Hamilton-Jacobi-Bellman equation for  the stochastic control problem (\ref{defvalfunc}) reads: For $z=(x,y)$ and all $(t,\wlth,z)\in[0,T)\times (0,\infty)\times E$
\begin{equation}
\label{vHJB1}
\begin{split}
&v_t(t,\wlth,z)+\tilde\mu_1(t,z)'v_x(t,\wlth,z)+\mu_2(t,y)'v_y(t,\wlth,z)+\frac12\trace(\sigma(t,z)'v_{xx}(t,\wlth,z)\sigma(t,z)) \\
&+\trace(\sigma_2(t,z)'v_{xy}(t,\wlth,z))+\frac12\Delta_yv(t,\wlth,z)
+\sup_{\pi\in\R^n}\big(\pi'\mu_1(t,z)\wlth v_\wlth(t,\wlth,z) \\
&+\frac12|\pi'\sigma(t,z)|^2\wlth^2v_{\wlth \wlth}(t,\wlth,z) +\pi'\sigma(t,z)\sigma(t,z)'\wlth v_{\wlth x}(t,\wlth,z)+\pi'\sigma_2(t,z)\wlth v_{\wlth y}(t,\wlth,z)\big)=0,
\end{split}
\end{equation}
and%
\begin{equation} \label{vTx}
v(T,\wlth,z)=\frac{\wlth^a}a.
\end{equation}

As usually, the supposed homogeneity property (\ref{eq: homogeneity a}) leads to the following ansatz:
\begin{equation}
\label{candsol}
v(t,\wlth,z)=\frac{\wlth^a}ae^{-u(t,z)}, \text{ for }   (t,\wlth,z)\in[0,T]\times (0,\infty)\times E.
\end{equation}
Under the hypothesis of this ansatz, we obtain that
\[
 \Big(\wlth \frac{\partial}{\partial \wlth}\Big)^n v(t,\wlth,z)=a^n v(t,\wlth,z), \text{ for } n \in \N.
\]
Insertion of this and  (\ref{candsol}) into equation (\ref{vHJB1}) gives, ignoring the argument $(t,z)$ in $u(t,z)$:  For $z=(x,y)$ and all $(t,z)\in[0,T) \times E$
\begin{equation}
\label{vHJB2}
-u_t- \frac12\trace(\sigma(t,z)'u_{xx}\sigma(t,z))  -\trace(\sigma_2(t,z)'u_{xy})  -\frac12\Delta_y u
+H(t,z,u_z)=0
\end{equation}
and
\begin{equation}
\label{vHJB2tc}
u(T,z)=0,
\end{equation}
where the function $[0,T] \times E \times E \ni (t,(x,y),(p,q))=(t,z,r) \mapsto H(t,z,r) \in   \R$ is defined by
\begin{equation}
\begin{split}
\label{def:function:H}
&H(t,z,r)=\frac12|\sigma(t,z)'p|^2 +q'\sigma_2(t,z)'p +\frac12|q|^2-\tilde\mu_1(t,z)'p-\mu_2(t,z)'q\\
&+a\max_{\pi\in\R^n}\left(\pi'(\mu_1(t,z)-\sigma(t,z)\sigma(t,z)'p-\sigma_2(t,z)q)-\frac{1-a}2|\sigma(t,z)'\pi|^2\right).
\end{split}
\end{equation}
Since the PDE (\ref{vHJB2}) is linear in the second order derivatives, it is by definition  semilinear.

Under the condition (\ref{cond: M1(z) inv}), $M(t,z)$ is invertible for all $(t,z)  \in [0,T] \times E$, so the unique solution  $\tilde{\pi}$ to the maximization problem in formula (\ref{def:function:H}) is explicitly given by
\begin{equation} \label{eq: opt control}
 \tilde{\pi}(t,z,r)=\frac{1}{1-a}M(t,z)^{-1}(\mu_1(t,z)-M(t,z)p-\sigma_2(t,z)q).
\end{equation}
Substitution into (\ref{def:function:H}) gives, for all $(t,z)  \in [0,T] \times E$ the second degree polynomial $H(t,z,r)$  in $r$:
\begin{equation}
\begin{split}
\label{eq: 2nd degree poly:H}
& H(t,z,r)
=\frac12\frac1{1-a}|\sigma(t,z)'p|^2+\frac1{1-a}q'\sigma_2(t,z)'p+\frac12q'\Big(\mathrm{I}_d+ \frac{a}{1-a}\sigma_2(t,z)'M(t,z)^{-1}\sigma_2(t,z)\Big)q\\
& +p'\Big(\beta(t,z)-\frac1{1-a}\mu_1(t,z)\Big)  -q'\Big(\mu_2(t,z)+\frac{a}{1-a} \sigma_2(t,z)'M(t,z)^{-1} \mu_1(t,z)\Big)\\
&+ \frac12\frac{a}{1-a}\mu_1(t,z)'M(t,z)^{-1}\mu_1(t,z).
\end{split}
\end{equation}
For later reference we note that
\begin{equation} \label{eq: sup sigma' hat pi}
\begin{split}
&|\sigma(t,z)'\tilde{\pi}(t,z,r) |^2 \\
&=(1-a)^{-2}(\mu_1(t,z)-M(t,z)p-\sigma_2(t,z)q)'M(t,z)^{-1}(\mu_1(t,z)-M(t,z)p-\sigma_2(t,z)q).
\end{split}
\end{equation}

In order to study the existence of classical solutions to the semilinear equation (\ref{vHJB2}) we introduce the linear space by $C^{1,2}_l$  of all functions
$u \in   C^0([0,T] \times E) \cap C^{1,2}([0,T) \times E)$
satisfying the following growth condition: With $z=(x,y)$, there exists $C \in \R$ such that
\begin{equation} \label{eq: cond grad u}
 \forall (t,z)\in [0,T)\times E, \, |M(t,z)^{1/2} u_x(t,z)|+ |u_y(t,z)|\leq C(1+|y|).
\end{equation}

The remaining part of this section is devoted to prove the existence result, of a classical solution to the semilinear HJB equation, formulated by
\begin{theorem}\label{thm:global}
Assume that Condition \ref{cond: A}, Condition \ref{cond: B}  and  Condition \ref{cond: C}  are satisfied Then there exists a solution $u \in C^{1,2}_l$ to the semilinear equation (\ref{vHJB2}) with the terminal condition (\ref{vHJB2tc}).
\end{theorem}

In order to prove Theorem \ref{thm:global}, we  reformulate the semilinear equation (\ref{vHJB2}), with the terminal condition (\ref{vHJB2tc}), as a stochastic control problem.

Supposing (\ref{cond: M1(z) inv}), we can re-write   for all $(t,z)\in [0,T)\times E$, the convex second degree polynomial  $E \ni r \mapsto H(t,z,r)$, given by (\ref{eq: 2nd degree poly:H}), on the following form:
\begin{equation}\label{explicit:form 2: H}
 H(t,z,r)=\frac{1}{2}(r,A(t,z)r) - (r,l(t,z)) + k(t,z),
\end{equation}
where, with $r=(p,q)$,
\begin{align}
 N(t,z)&=\sigma_2(t,z)'M(t,z)^{-1}\sigma_2(t,z), \label{def:N}\\
A(t,z)r&=\Big(\frac{1}{1-a}M(t,z)p+\frac{1}{1-a} \sigma_2(t,z)q,  \frac{1}{1-a} \sigma_2(t,z)' p + q + \frac{a}{1-a}N(t,z)q\Big), \label{def: A(z)} \\
\ell(t,z)&=(\ell_1(t,z),\ell_2(t,z))=\Big(\frac1{1-a}\mu_1(t,z)-\beta(t,z), \mu_2(t,y)+ \frac a{1-a}\sigma_2(t,z)'M(t,z)^{-1}\mu_1(t,z)\Big), \label{def: ell(t,z)} \\
k(t,z)&=\frac12\frac a{1-a}\mu_1(t,z)'M(t,z)^{-1}\mu_1(t,z). \label{def:k}
\end{align}
One checks that the symmetric positive linear operator $A(t,z)$ is positive definite when (\ref{cond: M1(z) inv}) is satisfied.

For all $(t,z)\in [0,T)\times E$, the convex function $L(t,z, \cdot)$ on $E$ is defined by the Legendre-Fenchel transformation  (in $-\bar r$) of the function $\bar r  \mapsto H(t,z,\bar{r})$:
\begin{equation}
\label{def:func:L}
L(t,z,r)=\sup_{\bar r \in E}(-(\bar{r},r)- H(t,z,\bar{r})).
\end{equation}
When  $A(t,z)$ is positive definite, which is supposed in the sequel by imposing (\ref{cond: M1(z) inv}), then $L(t,z,r)$  have the following explicit form:
\begin{equation}
\label{form:de:L:0}
L(t,z,r)=\frac{1}{2}(r-\ell(t,z))'A(t,z)^{-1}(r-\ell(t,z))-k(t,z).
\end{equation}
By convexity
\begin{equation}
\label{duality:relation:H}
H(t,z,r)=\sup_{\bar r \in E}(-(\bar{r},r)- L(t,z,\bar{r})),
\end{equation}
where the supremum is realized for $\bar{r}=\hat{r}(t,z,r)$,
\begin{equation}
\label{def:opt control :Hk}
\hat{r}(t,z,r)=\ell(t,z)-A(t,z)r.
\end{equation}

A  stochastic control problem, corresponding to the semilinear equation (\ref{vHJB2}), with the terminal condition (\ref{vHJB2tc}), is now
\begin{equation}
\label{def: control pr for u}
u(t,z)=\inf_{\nu \in\Uc_t}\E\Bigg[\int_t^T L(s,Z_s,\nu_s)\,ds\Bigg|Z_t=z\Bigg],
\end{equation}
where $\Uc_t$ is the set of all square integrable progressively measurable $E$-valued process independent of $\Fc_t$.

For $R  >0$, a regularization $H^R(t,z,r)$  of the  Hamiltonian $H(t,z,r)$ at $(t,z,r) \in [0,T] \times E^2$ is defined by
\begin{equation}
\label{def:func:Hk}
H^R(t,z,r)=\sup_{|\bar r| \leq R}(-(\bar{r},r)- L(t,z,\bar{r})).
\end{equation}
We consider then the HJB equation with the Hamiltonian $H$ replaced by $H^R$,
\begin{equation}
\begin{split}
 \label{HJB:on:u_k}
-u^R_t- &\frac12\trace(\sigma(t,z)'u^R_{xx}\sigma(t,z))  -\trace(\sigma_2(t,z)'u^R_{xy})  -\frac12\Delta_y u^R
+H^R(t,z,u^R_z)=0, \\
&\text{with final data }  u^R(T,\cdot)=0.
\end{split}
\end{equation}
We shall see that, \cite[Theorem 6.2 Chap VI]{Flem Rish 75} permits to deduce the existence of a $C^{1,2}$ solution   $u^R$.

We  shall use simple estimates for $A(t,z)$ and  $A(t,z)^{-1}$ (when it exists) and their first derivatives. The expression (\ref{def: A(z)}) gives directly
\begin{equation} \label{estimate A(z)}
 |A(t,z)|\leq 1+\frac{2}{1-a}(|M(t,z)|+|N(t,z)|), \;\; (t,z)  \in [0,T] \times E.
\end{equation}
and
\begin{equation} \label{estimate grad A(z)}
 |A_z(t,z)|\leq \frac{2}{1-a} (1+|M(t,z)^{1/2}|+2|M(t,z)^{-1/2}|)|\sigma_z(t,z)|, \;\; (t,z)  \in [0,T] \times E.
\end{equation}
The following explicit expression of $A(t,z)$ (cf. Schur complement) is convenient for estimating $A(t,z)^{-1}$:
\begin{equation} \label{eq: A(z) explicit 1}
A(t,z)=T(t,z)D(t,z)T(t,z)',
\end{equation}
where
\begin{equation} \label{eq: A(z) explicit 2}
T(t,z)r=(p,  \sigma_2(t,z)' M(t,z)^{-1}p+ q) \text{ and } D(t,z)r=(\frac{1}{1-a}M(t,z)p,q -N(t,z)q) .
\end{equation}
Since
\[
 |D(t,z)^{-1}| \leq (1-a) |M(t,z)^{-1}| + |(\mathrm{I}-N(t,z))^{-1}| \leq (1-a) |M(t,z)^{-1}| + \frac{1}{1-|N(t,z)|}
\]
and
\[
  |(T(t,z))^{-1}|\leq 1+|M(t,z)^{-1/2}|
\]
it follows that
\begin{equation} \label{estimate A(z) inv}
 |A(t,z)^{-1}(t,z)| \leq 2\left((1-a) |M(t,z)^{-1}| + \frac{1}{1-|N(t,z)|}\right)  \left(1+|M(t,z)^{-1}|\right).
\end{equation}
A direct calculation gives that
\begin{equation} \label{estimate grad N(z)}
 |N_z(t,z)|\leq 4 |M(t,z)^{-1/2}| |\sigma_z(t,z)|.
\end{equation}
Inequalities (\ref{estimate grad A(z)}), (\ref{estimate A(z) inv}) and (\ref{estimate grad N(z)}) give
\begin{equation} \label{estimate grad A(z) inv}
 |\nabla_z A(t,z)^{-1}(t,z)| \leq  \frac{32}{1-a} \left(1+|M(t,z)^{1/2}|+|M(t,z)^{-1}| + \frac{1}{1-|N(t,z)|}\right)^5  |\sigma_z(t,z)|.
\end{equation}
\begin{lemma}\label{lemma:existence u^R}
Let (\ref{eq: 2 lip1}) and (\ref{eq: 3 lip1}) of Condition \ref{cond: A} be satisfied  and let  $\sigma$ satisfy (\ref{cond: sigma C2}), (\ref{cond: sigma bound}), (\ref{cond: M1(z) inv}) and (\ref{eq: cond A invert}), then equation (\ref{HJB:on:u_k}) has a solution $u^R \in C^{1,2}([0,T)\times E)\cap C^0([0,T]\times E)$ being unique in the subset of such functions with polynomial growth.
\end{lemma}
\textit{Proof}: Let the  assumptions of the lemma be satisfied. Then, by Remark \ref{rmk: vol sqr M} and Lemma \ref{lm: derived conditions}, we can suppose without restriction that $m=n$ and that $\sigma$ is invertible. The existence and uniqueness $u^R$ then follows from \cite[Theorem 6.2 Chap VI]{Flem Rish 75}. In fact the hypothesis of this theorem are satisfied:
\begin{enumerate}
 \item The functions $\mu_1$, $\tilde\mu_1$, $\beta$, $\mu_2$ and $\sigma$ are $C^1$ with bounded derivative according to (\ref{eq: 2 lip1}) and (\ref{cond: sigma bound}).  It then follows that  there exists $C \in \R$ such that for all $z=(x,y) \in E$ and $f \in \{\tilde{\mu}_1, \mu_1, \beta,  \sigma\}$
\begin{equation} \label{estimate coeff 1}
 |f(t,z)| \leq C (1+|z|), \;\;  |f_z(t,z)| \leq C, \;\;   |\mu_2(t,y)| \leq C (1+|y|) \text{  and } |\nabla_y\mu_2(t,y)|\leq C.
\end{equation}
\item By (\ref{estimate coeff 1}), (\ref{cond: sigma C2}) and (\ref{cond: sigma bound}), $[0,T] \times E \ni (t,z) \mapsto \Sigma(t,z)$ is $C^2$ and bounded together with its first derivative. Then by  (\ref{cond: M1(z) inv}),  $\Sigma(t,z)$ is the invertible and $E \ni z \mapsto \Sigma(t,z)^{-1}$ is bounded together with its first derivative. So with the definition $\theta(t,z,r)=\Sigma(t,z)^{-1}r$,  points a) and b) of assumption (6.9) of \cite[Theorem 6.2 Chap VI]{Flem Rish 75} are satisfied.

\item By condition (\ref{eq: 2 lip1}) and  conditions  (\ref{cond: sigma bound}), (\ref{cond: M1(z) inv}) and (\ref{eq: cond A invert}), inequalities (\ref{estimate A(z) inv}) and (\ref{estimate grad A(z) inv}) give for some $C \in \R$
\begin{equation} \label{estimate A and grad A}
  |A(t,z)^{-1}(t,z)| \leq C \text{  and  }   |\nabla_z A(t,z)^{-1}(t,z)| \leq C |\sigma_z(t,z)| \leq C C_1,
\end{equation}
where the last inequality follows from (\ref{estimate coeff 1}) for some $C_1$ independent of $z$.

According to inequality (\ref{estimate A and grad A})  there exists $C>0$ such that for all $(t,z,r) \in [0,T] \times E \times E$
\begin{align}
  &|L(t,z,r)|  \leq C |r-\ell(t,z)|^2+|k(t,z)|, \label{estimate L} \\
  &|L_z(t,z,r)|\leq C |r-\ell(t,z)|^2+C |\ell_z(t,z)||r-\ell(t,z)|+|k_z(t,z)| . \label{estimate grad  L 1}
\end{align}
 By (\ref{estimate coeff 1}), (\ref{cond: M1(z) inv}) and (\ref{eq: cond A invert}) there is $C>0$ such that
for all $(t,z)  \in [0,T] \times E$
\begin{equation} \label{estimate k l}
  |k(t,z)| +  |l(t,z)|^2 \leq C (1+|z|)^2 .
\end{equation}
Using now (\ref{estimate coeff 1}), (\ref{estimate grad N(z)}), (\ref{cond: sigma bound}), (\ref{cond: M1(z) inv}), (\ref{eq: cond A invert}) it follows that the $C$ in (\ref{estimate k l}) can be chosen such that
\begin{equation} \label{estimate grad k l}
  |k_z(t,z)| +  |l_z(t,z)|^2 \leq C (1+|z|)^2 .
\end{equation}

The above estimates for $L$, $l$ and $k$ and their first derivatives imply  %
that the $C^1$ function $[0,T] \times E \times \bar B_E(0,R) \ni (t,z,r) \mapsto L(t,z,r)$ is bounded together with its first derivative by a second degree polynomial in $z$. This shows that also point c) of assumption (6.9) of \cite[Theorem 6.2 Chap VI]{Flem Rish 75} is satisfied. Therefore all the assumptions of \cite[Theorem 6.2 Chap VI]{Flem Rish 75} are true.
\end{enumerate}
\qed

We have the following estimate, uniform in $R$, for the derivative $u^R_z$.
\begin{lemma}
\label{lemma:linear:growth:u^k}
Let the conditions (\ref{eq: growth}), (\ref{cond: bound mu1 2}) and (\ref{cond: bound mu1 3}) and the conditions of Lemma \ref{lemma:existence u^R}  be satisfied. Then there exists  $C \in \R$ such that for  all $(t,z)=(t,(x,y)) \in [0,T] \times E$ and  $R >0$
\[
 |u^R_z(t,z)|\leq C(1+|y|).
\]
\end{lemma}
\noindent\textit{Proof}:
The solution $u^R$ to (\ref{HJB:on:u_k}) has a stochastic control representation, which solution is obtained by verification, cf. \cite[Theorem 6.4 Chap VI]{Flem Rish 75} and \cite[Th ??]{Touzi OSC 2013}
\begin{equation}
\label{def:of:u^k(t,s,y)}
u^R(t,z)=\inf_{\nu \in\Uc_t(R)}\E^\Q\left[\int_t^T L(s,\bar Z_s,\nu_s)\,ds\Big|\bar Z_t=z\right],
\end{equation}
where $\Uc_t(R) =  \{\nu \in \Uc_t \; : \:  |\nu| \leq R  \text{ a.e. } dt\, d\Q\}$  and where the controlled dynamics of $\bar{Z}$ is given by
\begin{equation}\label{dZt:form:3}
 d\bar{Z}_t=\nu_t \,dt + \Sigma(\bar{Z}_t)\,d \bar{W}^\Q_t
\end{equation}
with  $\bar{W}^\Q$ a $(m+d)$-dimensional standard Brownian motion under $\Q$ and $\bar{W}^\Q=(B^\Q,W^\Q)$, ${W}^\Q$ being $d$-dimensional.

Let $\hat{Z}$ be the solution of (\ref{dZt:form:3}) with the control $\nu_t=\hat{r}_R(\bar{Z}_t, u^R_z(t,\bar{Z}_t))$, where the optimal control function $\hat{r}_R$ is given by (\ref{def:opt control :Hk}).
Then the optimal control $\hat{\nu}$ for (\ref{def:of:u^k(t,s,y)}) is given by $\hat{\nu}_t=\hat{r}_R(\hat{Z}_t, u^R_z(t,\hat{Z}_t))$ and by (\ref{def:of:u^k(t,s,y)}),
\begin{equation}
\label{def:of:u^k(t,s,y):2}
u^R(t,z)=\E^\Q\left[\int_t^T L(\hat Z_s,\hat{\nu}_s)\,ds\Bigg| \hat Z_t=z\right].
\end{equation}
By the controlled SDE (\ref{dZt:form:3}) and results from \cite{Flem Soner 1993}, the derivative of $u^R$ is given by (see Appendix Corollary \ref{cor: A} for a proof):
\begin{equation}
\label{formula:for:u^k_z(t,s,y)}
u^R_z(t,z)=\E^\Q\left[\int_t^T L_z(\hat Z_s,\hat{\nu}_s)\,ds\Bigg| \hat Z_t=z\right].
\end{equation}

The function  $k_z$ is a bounded by (\ref{cond: bound mu1 2}). The functions $\nabla\mu_1$, $\nabla\mu_2$ and  $\nabla\beta$ are bounded by  (\ref{estimate coeff 1}) and the function  $\nabla (\sigma_2'M^{-1}\mu_1)$ is bounded by (\ref{cond: bound mu1 3}). This shows that the function $\ell_z$ is bounded. It now follows from inequality (\ref{estimate grad  L 1}), with a new constant $C$ and for some positive  constant $C_1$, that for all $(t,z,r) \in   [0,T] \times E \times E$
\begin{equation}\label{estimate grad  L 2}
 |L_z(t,z,r)| \leq C_1 |r-\ell(t,z)|^2+C_1 |\ell_z(t,z)||r-\ell(t,z)|+|k_z(t,z)| \leq C(1 + |r-\ell(t,z)|^2).
\end{equation}
By conditions (\ref{cond: sigma bound}) and (\ref{eq: cond A invert}) and by inequality (\ref{estimate A(z)}) the function $ [0,T] \times E \ni (t,z) \mapsto |A(t,z)|$ is bounded by a constant  $C >0$, so $|A(t,z)^{-1}| \geq C^{-1}$. Hence with a new constant $C$
\[
  |L_z(t,z,r)| \leq C\left(1 + \frac{1}{2}(r-\ell(t,z), A(t,z)^{-1}(r-\ell(t,z))\right) = C\left(1 + L(t,z,r) +k(t,z) \right) .
\]
The function $k$ is bounded according to  (\ref{cond: bound mu1 2}). Let  $c(k)$ be a bound. Then, for all $(t,z,r) \in E \times E$,
 \begin{equation}\label{estimate grad  L 3}
 |L_z(t,z,r)| \leq C\left(1 + L(t,z,r) + c(k)\right) .
\end{equation}

Let $\bar Z^0$ be a solution of the SDE (\ref{dZt:form:3}) with $\nu=0$. Inequality (\ref{estimate grad  L 3}), the  stochastic control representations (\ref{def:of:u^k(t,s,y)}) and (\ref{formula:for:u^k_z(t,s,y)}) and inequality (\ref{estimate L}) then give, for some positive constants $C_1$ independent of $(t,z)$ and $R$, that
\begin{equation}\label{estimate grad uR  1}
 \begin{split}
|u^R_z(t,z)|&=\Bigg|\E^\Q\left[\int_t^T L_z(s,\hat Z_s,\hat{\nu}_s)\,ds\Big| \hat Z_t=z\right]\Bigg| \\
  &\leq \E^\Q\left[  \int_t^T  {C\left(1 + L(s,\hat Z_s,\hat{\nu}_s)+c(k)\right) \,ds\Big|\hat Z_t=z}\right]\\
 & \leq C\left(C_1+\E^\Q\left[\int_t^T L(s,\hat Z_s,\hat{\nu}_s)\,ds\Big|\hat Z_t=z\right]\right)\\
 & \leq C\left(C_1+\E^\Q\left[\int_t^T L(s,\bar Z^0_s,0)\,ds\Big|\bar Z^0_t=z\right]\right).
 \end{split}
\end{equation}
Due to conditions (\ref{eq: growth}) and (\ref{cond: bound mu1 3}), there is $C \in \R$ such that for all $(t,z)=(t,(x,y)) \in  [0,T] \times E$,
\begin{equation}\label{square growth cond on ell}
  |\ell(t,z)|^2 \leq C (1+|y|).
\end{equation}

The estimate (\ref{estimate L}) and $|k(t,z)| \leq c(k)$ then give, with a new constant $C$, that for all $(t,z) \in  [0,T] \times E$,
\[
 |L(t,z,0)| \leq C (1+|y|).
\]
It now follows from (\ref{estimate grad uR  1}) that, for some constant $C_2$ independent of $(t,z)$ and $R$
\[
 |u^R_z(t,z)| \leq C \Bigg(C_1+C_2\E^\Q\left[\int_t^T (|y|+|W^\Q_s|)\,ds\right]\Bigg).
\]
This proves the existence of $C \in \R$ such that for all $(t,z,R) \in  [0,T] \times E \times (0, \infty)$,
\begin{equation}\label{estimate grad uR  2}
|u^R_z(t,z)| \leq C (1+|y|), \text{  where } z=(x,y).
 \end{equation}

\qed

\begin{proof}[Proof of Theorem \ref{thm:global}]
The function $E \ni r\rightarrow -(r,u_z^R(t,z))_E-L(t,z,r)$ (with $L$ as in (\ref{def:func:L})) attains its maximum on $E$ for

$$\hat r^R(t,z)=\ell(t,z)-A(t,z)u^R_z(t,z).$$

Recall that $A(t,z)$ is bounded by (\ref{estimate A(z)}) and $\ell$ satisfies (\ref{square growth cond on ell}). Moreover, by Lemma \ref{lemma:linear:growth:u^k} there is a $C$ such that, for all $(t,z,R) \in [0,T] \times E \times (0,\infty)$,  $u^R_z$ satisfies the inequality $|u_z^R(t,z)|\leq C(1+|y|)$. Hence for given $c \geq 0$, there exists a positive constant $C_1$ independent of $R$ such that
$$|\hat r^R(t,z)| \leq C_1,\;\;\forall \; (t,z) \in [0,T] \times E \text{ such that }|y|\leq c.$$
Therefore, for $R \geq C_1$, we have
\begin{equation}
\begin{split}
 H^R(t,z,u^R(t,z))&=\sup_{\bar{r}\in\bar{B}_E(0,R)}\left[-\bar{r}'u^R(t,z)-L(t,z,\bar{r})\right]\\
 &=\sup_{\bar{r}\in E}\left[-\bar{r}'u^R(t,z)-L(t,z,\bar{r})\right]\\
 &=H(t,z,u^R(t,z)),
\end{split}
\end{equation}
for all $(t,z) \in [0,T] \times E$ such that $|y|\leq c$. Since  $c$ can be chosen arbitrarily large, this implies that $u^R$ is a  $C^{1,2}$ solution to (\ref{vHJB2})-(\ref{vHJB2tc}) satisfying  (\ref{eq: cond grad u}).

\end{proof}

\section{Existence of a classical solution to the stochastic optimal control problem}
This section is devoted to proving the main result of this note, Theorem \ref{thm: main}, which is a trivial corollary (not stated formally) of the following Proposition \ref{prop:verification} and of the existence of a classical solution in Theorem \ref{thm:global}. In Proposition \ref{prop:verification} a verification result, which relates a solution of the semilinear equation (\ref{vHJB2}), with terminal condition (\ref{vHJB2tc}), to the original stochastic control problem (\ref{defvalfunc}), is proved using elementary properties of the Girsanov transformation (cf.   \cite{Pham 2002}).
\begin{prop}\label{prop:verification}
If  $u\in C^{1,2}_l$ is a solution of  the semilinear PDE (\ref{vHJB2}), with the terminal condition (\ref{vHJB2tc}),  if the assumptions (\ref{eq: 3 lip1}) and  (\ref{eq: 2 lip1}) are satisfied and if %
\begin{equation} \label{eq: bound mu1}
[0,T] \times E \ni (t,z)=(t,(x,y)) \mapsto  |M(t,z)^{-1/2}  \mu_1(t,z)|(1+|y|)^{-1} \text{ is a bounded function}
\end{equation}
then the value function of (\ref{defvalfunc}) is given by
\[
 v(t,\wlth,z)=\frac{\wlth^a}{a} e^{-u(t,z)}, \text{ for } (t,\wlth,z)\in[0,T]\times (0,\infty) \times E.
\]
There is a unique optimal  control process $\hat{\pi} \in \Ac$,
\[\hat{\pi}_t= \tilde{\pi}(t,Z_t,u_z(t,Z_t)), \]
where  $\tilde{\pi}$  is defined by (\ref{eq: opt control}).
\end{prop}
\noindent\textit{Proof}:

1. We first prove that $\hat{\pi} \in \Ac$. With norms in relevant linear spaces, it follows from (\ref{eq: sup sigma' hat pi}) that
\begin{equation} \label{eq: sup sigma' hat pi estimate}
\begin{split}
& |\sigma(t,z)'\tilde{\pi}(t,z,r) |^2 =\frac{1}{(1-a)^2} |M(t,z)^{-1/2} (\mu_1(t,z)-M(t,z)p-\sigma_2(t,z)q)|^2 \\
& \leq \frac{3}{(1-a)^2} ( |M(t,z)^{-1/2} \mu_1(t,z)|^2 + |M(t,z)^{1/2}p|^2 +  |M(t,z)^{-1/2} \sigma_2(t,z)q)|^2).
\end{split}
\end{equation}
Using conditions (\ref{eq: bound mu1}) %
and (\ref{eq: cond grad u}), this gives for some $C \in \R$ that
\begin{equation} \label{eq: sup sigma' hat pi estimate 2}
\begin{split}
& |\sigma(t,z)'\hat{\pi}(t,z,u_x(t,z),u_y(t,z)) |^2 \\
& \leq \frac{3}{(1-a)^2} ( |M(t,z)^{-1/2} \mu_1(t,z)|^2 + |M(t,z)^{1/2}u_x(t,z)|^2 +  |M(t,z)^{-1/2} \sigma_2(t,z)u_y(t,z))|^2)\\
& \leq C (1+|y|^2).
\end{split}
\end{equation}
It now follows from (\ref{supEYt}) that $\hat\pi$ satisfies (\ref{intcon1}), so  $\hat\pi \in \Ac$.

2. Let $\pi \in \Ac$  be  an admissible control. We define the measure $\Q^\pi$ by
\begin{equation} \label{eq: xi pi}
\frac{d\Q^\pi}{d\PP}=\exp\left(\int_0^T a\pi_s'\sigma(s,Z_s)\,d \tilde{W}_s
-\frac12\int_0^T |a\sigma(s,Z_s)'\pi_s|^2\,ds \right).
\end{equation}
Since  $\pi \in \Ac$, condition (\ref{intcon1})  implies that $\Q^\pi$ is a probability measure (cf. Exercise 1.40, \S1, Ch.VIII \cite{Revuz-Yor}). %
The   $m$ and $d$-dimensional processes  $B^\pi$    and $W^\pi$ respectively and the  $m+d$-dimensional process $\tilde{W}^\pi=(B^\pi,W^\pi)$, are defined by
\[
 d\tilde{W}^\pi_t=d\tilde{W}_t-a\sigma(t,Z_t)'\pi_t\,dt, \;\; \tilde{W}^\pi_0=0.
\]
According to Girsanov's theorem these three processes are standard multi-dimensional $\Q^\pi$-B.m. and
\begin{equation}
\label{dStdYtGir}
\begin{split}
dX_t&=(\tilde{\mu}_1(t,Z_t)+a\sigma(t,Z_t)\sigma(t,Z_t)'\pi_t)\,dt+\sigma(t,Z_t)\,d\tilde{W}^\pi_t,\\
dY_t&=(\mu_2(t,Y_t)+a\sigma_2(t,Z_t)'\pi_t)\,dt+dW^\pi_t,
\end{split}
\end{equation}

3.  If  $u \in C^{1,2}_l$ is a solution to (\ref{vHJB2}) and (\ref{vHJB2tc}) and $\pi \in \Ac$ , then  an exponential $\Q^\pi$-local martingale $\xi^\pi$, is defined by
\begin{equation}
\begin{split}
&\xi_t^\pi=\exp\biggl(-\int_0^t (u_x(s,Z_s)'\sigma_1(s,Z_s)\,dB_s^\pi +(u_y(s,Z_s)'+u_x(s,Z_s)'\sigma_2(s,Z_s))\,dW_s^\pi) \\
&-\frac{1}{2}\int_0^t(|\sigma_1(s,Z_s)'u_x(s,Z_s)|^2+|u_y(s,Z_s)+\sigma_2(s,Z_s)'u_x(s,Z_s)|^2)\,ds\biggr).
\end{split}
\end{equation}
The particular case $\xi^{\hat\pi}$ is a $\Q^\pi$-martingale. In fact, according to inequality (\ref{eq: cond grad u}) there exists $C \in\R$ such that for all $t \in [0,T)$
\[
|u_x(t,Z_t)'\sigma_1(t,Z_t)|^2+|u_y(t,Z_t)'+u_x(t,Z_t)'\sigma_2(t,Z_t)|^2 \leq C (1+|Y_t|^2).
\]
The second SDE in (\ref{dStdYtGir}), with $\pi=\hat\pi$, and inequality (\ref{eq: sup sigma' hat pi estimate 2}) then imply that the above constant $C$ can be chosen such that
\[
 |Y_t| \leq |Y_0| +\int_0^t (|\mu_2(s,Y_s)|+a|\sigma_2(s,Z_s)'\hat\pi_s|)\,ds + |W^{\hat\pi}_t| \leq |Y_0| +C\int_0^t (1+|Y_s)|)\,ds + |W^{\hat\pi}_t|,
\]
which using Gr\"{o}nwall's inequality leads to the existence of $\veps >0$ such that $$\sup_{0\leq t\leq T}\E^{\Q^{\hat\pi}}    \left[e^{\veps|Y_t|^2}\right]< \infty.$$ The claim now follows by using the above mentioned exercise.

4. By (\ref{dWltht})
$$\Wlth_T=\wlth \exp\left(\int_t^T (\pi_s'\mu_1(s,Z_s)\,ds+\pi_s'\sigma(s,Z_s)\,d\tilde{W}_s) -\frac{1}{2} \int_t^T |\pi_s'\sigma(s,Z_s)|^2 \, ds \right), $$
which together with (\ref{eq: gain fnct}) gives
\begin{equation}
\label{formuleforJ}
J(t,\wlth,z,\pi)=\frac{\wlth^a}{a} \E^{\Q^\pi}\left[\exp\left(\int_t^T{l(s,Z_s,\pi_s)\,ds}\right)\biggr|Z_t=z \right],
\end{equation}
where for $\alpha \in \R^n$, $l(t,z,\alpha)=a\alpha'\mu_1(t,z)-(a(1-a)/2) |\sigma(t,z)'\alpha|^2.$

Now let  $u \in C^{1,2}_l$ be a solution to (\ref{vHJB2}) and (\ref{vHJB2tc}).  By formula (\ref{dStdYtGir}), It\^o's formula applied to $u(t,Z_t)$ gives:
\begin{equation}
\label{ItoDeu(T,Z_T)}
\begin{split}
&u(T,Z_T)=u(t,Z_t)+\int_t^T\biggl(u_t(s,Z_s)+u_x(s,Z_s)'(\tilde{\mu}_1(s,Z_s)+a\sigma(s,Z_s)\sigma(s,Z_s)'\pi_s)\\
&+u_y(s,Z_s)'(\mu_2(s,Y_s)+a\sigma_2(s,Z_s)'\pi_s)+\frac12\trace(\sigma(s,Z_s)'u_{xx}(s,Z_s)\sigma(s,Z_s))\\
&+\frac12\Delta_yu(s,Z_s)+\trace(\sigma_2(s,Z_s)'u_{xy}(s,Z_s))\biggr)\,ds\\
&+\int_t^T\biggl( u_x(s,Z_s)'\sigma_1(s,Z_s)\,dB_s^\pi + (u_y(s,Z_s)'+u_x(s,Z_s)'\sigma_2(s,Z_s))\,dW_s^\pi \biggl)\geq u(t,Z_t)\\
&+\int_t^T{l(s,Z_s,\pi_s)\,ds}+\int_t^T \biggl(u_x(s,Z_s)'\sigma_1(s,Z_s)\,dB_s^\pi
+(u_y(s,Z_s)'+u_x(s,Z_s)'\sigma_2(s,Z_s))\,dW_s^\pi\biggl) \\
&+\frac{1}{2}\int_t^T\biggl(|\sigma_1(s,Z_s)'u_x(s,Z_s)|^2 +  |u_y(s,Z_s)+\sigma_2(s,Z_s)'u_x(s,Z_s)|^2\biggr)\,ds.
\end{split}
\end{equation}
$\xi^\pi$ is a $\Q^\pi$-supermartingale and the terminal condition $u(T,z)=0$, so  it follows from (\ref{ItoDeu(T,Z_T)}) that, %
\begin{equation}
\label{expintl(Zt)}
\E^{\Q^\pi}\left[\exp\biggl(\int_t^T{l(s,Z_s,\pi_s)\,ds} \biggr)   \,\biggr| \, \Fc_t  \right]   \leq\exp(-u(t,Z_t))\E^{\Q^\pi}\left[\frac{\xi_T^\pi}{\xi_t^\pi}\biggr|\Fc_t\right]   \leq \exp(-u(t,Z_t)),  %
\end{equation}

Since  $\pi \in \Ac$ was  arbitrary up to now, it follows from (\ref{formuleforJ}) that %
\begin{equation}
\label{inequsurJ}
\forall \; \pi \in \Ac, \;\; J(t,x,z,\pi)\leq \frac{x^a}a\exp(-u(t,z)).
\end{equation}
Now choosing in particular $\pi = \hat \pi $, there is equality in (\ref{ItoDeu(T,Z_T)}). Since   $\xi^{\hat\pi}$ is a $\Q^\pi$-martingale, it then follows  that all the inequalities in  (\ref{expintl(Zt)}) and (\ref{inequsurJ}) are equalities, which proves that
\begin{equation*}
v(t,x,z)=\mathop{\sup}_{\pi\in\Ac} J(t,x,z,\pi)=J(t,x,z,\hat\pi)=\frac{x^a}ae^{-u(t,z)}.
\end{equation*}
\qed

\appendix
\section{Appendix} \label{app: A}

Let $E=\R^p$, $F=\R^q$  where $p,q\in\N^*$.  $W$ is a $q$-dimensional standard Brownian motion
 on a complete probability space $(\Omega,\Fc,\Q)$ endowed with the complete filtration $\F=(\Fc_t)_{t\geq 0}$ generated by $W$.
For $t\in[0,T]$, where $T>0$ is fixed, we denote by $\Uc_t^2$, the collection of all progressively measurable $E$-valued processes $\phi$ independent of $\Fc_t$, such that $\E\left[\int_t^T|\phi(s)|^2\,ds\right]<\infty$. For $R>0$, we set $$\Uc_t^2(R)=\{\nu\in\Uc_t^2 : |\nu|\leq R~a.e~dt\,d\Q\}.$$

Given $(t,z) \in [0,T] \times E$, we consider the following SDE for the process $Z^{t,z}$:
\begin{equation}
\label{Z^{t,z}(s):dynamic}
 Z^{t,z}(s)=z+\int_t^s{\nu(u)\,du}+\int_t^s{\sigma(u,Z^{t,z}(u))\,dW_u},~~t\leq s\leq T,
\end{equation}
where $ [0,T] \times E \ni (t,z) \mapsto \sigma(t,z) \in L(F,E)$ is a continuous function,  Lipschitz continuous in $z$ (with Lipschitz constant $K$ independent of $t$ and where the control process $\nu(\cdot)\in\Uc_t^2(R)$.

By classical theorems (cf. \cite[9. Theorem, p.83 and 10. Corollary, p.85]{Krylov1 1980}), there exist a strong solution and  a constant $C(q,K)$ such that for all $t \in [0,T]$, $z, z' \in E$ and $q\geq 2$
\begin{equation} \label{E sup Z carre majoree}
 \E\left[\sup_{t \leq s \leq T} |Z^{t,z}(s)|^{q}\right] \leq C(q,K)C'(1+|z|^{q}),
\end{equation}
where  $C'=1+R^q+\sup_{t\leq s \leq T} |\sigma(s,0)|$, and
\begin{equation} \label{inequality z Esup Z}
 \E\left[\sup_{t\leq s\leq T}|Z^{t,z'}(s)-Z^{t,z}(s)|^{q}\right]\leq  C(q,K) |z-z'|^{q}.
\end{equation}

Denote by $C^{0,1,0}([0,T] \times E \times E)$ the linear space of all real functions on $[0,T] \times E \times E$ such that $f$ and $f_z$ are continuous, where $f_z(t,z,v)=\nabla_z f(t,z,v)$.
Let $L:[0,T] \times E\times E\rightarrow \R$ be a function  satisfying the following conditions:
\begin{equation} \label{assumptions on L}
 \begin{split}
 &\textrm{(a)}~~L\in C^{0,1,0}([0,T] \times E \times E)\\
 &\textrm{(b)}~~ \exists k,  C_1 \geq 0   \text{  such that }  |L(t,z,v)| \leq C_1(1+|z|^k),~\forall (t,z,v) \in [0,T] \times E \times  \bar{B}(0,R),\\
 &\textrm{(c)}~~ \exists l, C_2 \geq 0\text{  such that } |L_z(t,z,v)| \leq C_2 (1+|z|^l),~\forall (t,z,v) \in [0,T] \times E \times  \bar{B}(0,R).
 \end{split}
\end{equation}
For $(t,z)\in[0,T]\times E$ fixed,  we consider the problem of minimizing
$$J(t,z;\nu)=\E\left[\int_t^T L(s,Z^{t,z}(s),\nu(s))\,ds\right],$$
in  $\nu \in \Uc_t^2(R)$.  Our goal is here to find the derivative of the function $u^R:[0,T]\times E \rightarrow \R$, with respect to the second argument, where $u^R$ is defined by
$$u^R(t,z)=\inf_{\nu\in\Uc_t^2(R)} J(t,z;\nu).$$
\begin{lemma}
If $L$ satisfies (\ref{assumptions on L}),   $\sigma$ is as in (\ref{Z^{t,z}(s):dynamic})  and  $\nu(\cdot)\in\Uc_t^2(R)$, then the derivative $J_z$ exists and
$$J_z(t,z;\nu)=\E\left[\int_t^T L_z(s,Z^{t,z}(s),\nu(s))\,ds\right].$$
\end{lemma}
\begin{proof}
For $h \in \R\setminus\{0\}$, $(t,z)\in[0,T]\times E$ and  $\xi\in E$, we shall prove that
$$\mathop{\lim}\limits_{h \rightarrow 0} \frac1h(J(t,z+h\xi;\nu)-J(t,z;\nu))=\E\left[\int_t^T L_z(s,Z^{t,z}(s) \nu(s))\,ds\right]\cdot\xi.$$
Since $L$ is $ C^{0,1,0}$ we have, for all $s\in [0,T]$ and  $z_1,z_2,v \in E$,
$$L(s,z_1,v) -L(s,z_2,v)=\left(\int_0^1 L_z(s,(1-\lambda)z_1 + \lambda z_2,v)\,d\lambda\right) \cdot (z_1-z_2).$$
This formula and the definition of $J$ give
\begin{equation}
 \begin{split}
 &J(t,z+h\xi;\nu)-J(t,z;\nu)\\
 & = \E\left[\int_t^T  \left(L(s,Z^{t,z+h\xi}(s),\nu(s))-L(Z^{t,z}(s),\nu(s))\right)\,ds\right]=\E[Y_h],
 \end{split}
\end{equation}
where 
$$Y_h=\int_t^T \left(\int_0^1 L_z(s,Z^\lambda(t,s),\nu(s))\cdot (Z^{t,z+h\xi}(s)-Z^{t,z}(s))\,d\lambda\right)\,ds,$$
with  $Z^\lambda(t,s)=(1-\lambda)Z^{t,z+h\xi}(s)+\lambda Z^{t,z}(s)$.

In order to study the behavior of $\E[Y_h]$ for small $|h|$, we rewrite $Y_h$ by using (\ref{Z^{t,z}(s):dynamic}):
\begin{equation} \label{eq: lem A1}
Y_h=U_h+V_h,
\end{equation}
where
\begin{equation} \label{eq: lem A1 2}
 U_h=(h\xi) \cdot \int_t^T \Lambda_h(t,s) \,ds, \;\; V_h=\int_t^T \Lambda_h(t,s)   \cdot M_h(t,s) \,ds,
\end{equation}
$$\Lambda_h(t,s)= \int_0^1  L_z(s,Z^\lambda(t,s),\nu(s))\,d\lambda$$
and
$$M_h(t,s)= \int_t^s{\left(\sigma(\tau,Z^{t,z+h\xi}(\tau))-\sigma(\tau,Z^{t,z}(\tau))\right)}\,dW_\tau .$$
Due to the Lipschitz property of $\sigma$
$$
\E[(M_h(t,s))^2]=\E\left[\int_t^s|\sigma(\tau,Z^{t,z+h\xi}(\tau))-\sigma(\tau,Z^{t,z}(\tau)|^2 \,d\tau\right] \leq K^2 \E\left[\int_t^s |Z^{t,z+h\xi}(\tau)-Z^{t,z}(\tau)|^2 \,d\tau\right] .
$$
Inequality (\ref{inequality z Esup Z}) then gives
$$
\E[(M_h(t,s))^2] \leq C (s-t) |h\xi|^2,
$$
for a $C \in \R$  independent of $s$, $t$, $z_1,z_2$ and  $v$.
According to $\textrm{(c)}$ of (\ref{assumptions on L}) we have:
\begin{equation*} \notag %
 |L_z(s,Z^\lambda(t,s),\nu(s))| \leq C_2(1+|Z^\lambda(s)|^l) \leq C_2 (1+|Z^{t,z}(s)|+|Z^{t,z+h\xi}(s)|)^l.
\end{equation*}
The estimate (\ref{E sup Z carre majoree}) then gives
$$
\E[(\Lambda_h(t,s))^2] \leq C(1+|z| + |z+h\xi|)^{2l},
$$
for a $C \in \R$  independent of $s$, $t$, $z_1,z_2$ and  $v$. So the Cauchy-Schwartz inequality gives
$$
\E[|V_h|] \leq \E\left[\int_t^T |\Lambda_h(t,s)\cdot M_h(t,s)| \,ds\right] \leq C  |h\xi| (1+|z| + |z+h\xi|)^{l},
$$
for a $C \in \R$  independent of $s$, $t$, $z_1,z_2$ and  $v$. Since $M_h(t,\cdot)$ is a martingale restricted to $[t,T]$, it now follows from (\ref{eq: lem A1 2}) and the Fubini theorem that
\begin{equation} \label{eq: lem A1 3}
 \E[V_h]=0.
\end{equation}

For all $\varepsilon>0$, we obtain by Markov's inequality and by (\ref{inequality z Esup Z}), with $q=2$, that
 $$\Q\left(\sup_{t\leq s \leq T}|Z^{t,z+h\xi}(s)-Z^{t,z}(s)|\geq \varepsilon\right)\leq \frac1{\varepsilon^2}\E\left[\sup_{t\leq s \leq T}|Z^{t,z+h\xi}(s)-Z^{t,z}(s)|^2\right]\leq \frac{Ch^2|\xi|}{\varepsilon^2},$$
where  $C \in \R$ is independent of $t$, $z_1,z_2$ and  $v$. So  $\sup_{t\leq s \leq T}|Z^{t,z+h\xi}(s)-Z^{t,z}(s)|$ converges to $0$ in probability, when $h\rightarrow 0$. By the continuity of the function
$$
[0,T] \times E \times E \times E \ni (s,z_1,z_2,z) \mapsto \int_0^1  L_z(s,(1-\lambda)z_1+\lambda z_2,v)\,d\lambda- L_z(s,z,v)
$$
it then follows that also $\sup_{t\leq s \leq T}|\Lambda_h(t,s)   -  L_z(s, Z^{t,z}(s),\nu(s))|$ converges to $0$ in probability, as $h\searrow0$. This is then also the case for $\int_t^T (\Lambda_h(t,s) -  L_z(s,Z^{t,z}(s),\nu(s)))  \,ds$.

By (c) of (\ref{assumptions on L}) and (\ref{E sup Z carre majoree}) it follows that the family
$$\left\{\int_t^T (\Lambda_h(t,s) -  L_z(s,Z^{t,z}(s),\nu(s)))\,ds\,:\, |h|\leq 1\right\}$$
 is uniformly integrable. This gives that
$$\lim_{h \rightarrow 0}\E\left[\left|\int_t^T (\Lambda_h(t,s) -  L_z(s,Z^{t,z}(s),\nu(s)))\,ds\right|\right]=0,$$
which, with (\ref{eq: lem A1}) and (\ref{eq: lem A1 3}), shows that
$$\lim_{h \rightarrow 0}\E\left[\left|\frac{1}{h} Y_h -  \xi \cdot L_z(s,Z^{t,z}(s),\nu(s)))\,ds\right|\right]=0.$$

\end{proof}

Given $(t,z)\in[0,T]\times E$, let $\nu^R_*(t,z)$ be the unique element in $\Uc_t^2(R)$ for which $\nu\mapsto J(t,z;\nu)$ takes its minimum on $\Uc_t^2(R)$. Now let $Z_*^{t,z}=(Z_*^{t,z}(s))_{t\leq s\leq T}$ be the solution to 

\begin{equation} \label{Zetoile}
 Z_*^{t,z}(s)=z+\int_t^s\nu_*^R(u,Z_*^{t,z}(u))\,du+\int_t^s\sigma(u,Z_*^{t,z}(u))\,dW_u, ~~t\leq s \leq T.
\end{equation}

\begin{coro} \label{cor: A}
 For every $(t,z)\in[0,T]\times E$,
 \begin{equation}
 \label{expression:of:D_x V(t,z)}
  u^R_z(t,z)=\E\left[\int_t^T L_z(s,Z_*^{t,z}(s),\nu_*^R(s,Z_*^{t,z}(s)))\,ds\right],
 \end{equation}
where $Z_*^{t,z}$ is the solution to (\ref{Zetoile}).
\end{coro}

\begin{proof}
 Since $\nu_*^R$ is an optimal Markov control policy,
 $$u^R(t,z)=J(t,z;\nu_*^R(t,z)).$$
 Given $(t,z)\in[0,T]\times E$, $h>0$ and $\xi\in E$ we obtain
 \begin{equation}
  \begin{split}
   \frac1h(u^R(t,z+h\xi)-u^R(t,z))&=\frac1h\Big(\inf_{\nu\in \Uc_t^2(R)} J(t,z+h\xi;\nu) - J (t,z;\nu_*^R(t,z))\Big)\\
   &\leq \frac1h(J(t,z+h\xi;\nu_*^R(t,z))-J(t,z;\nu_*^R(t,z)));
  \end{split}
 \end{equation}
and the right side tends to $J_z(t,z;\nu^R_*(t,z))\cdot \xi$ as $h\searrow 0$.
Therefore, by precedent lemma,
$$u^R_z(t,z)\cdot \xi \leq \int_t^T L_z(s,Z_*^{t,z}(s),\nu_*^R(s,Z_*^{t,z}(s)))\,ds \cdot \xi.$$
This holds for all directions $\xi$, in particular with $\xi$ replaced by $-\xi$, which gives (\ref{expression:of:D_x V(t,z)}).
\end{proof}

\end{document}